\documentclass[12pt,reqno,a4paper]{amsart}
\usepackage{amsmath,amssymb,amsfonts,amscd}
\usepackage[mathscr]{eucal}
\usepackage{hyperref,xcolor} 
\usepackage{comment,textcomp}

\def\co{\colon\thinspace}

\newtheorem{theorem}{Theorem}

\newtheorem{proposition}{Proposition}

\theoremstyle{definition}
\newtheorem{example}{Example}

\newcommand{\f}{{\varphi}}

\newcommand{\p}{\partial}

\newcommand{\g}{\gamma}

\newcommand{\der}[2]{{\frac{\partial {#1}}{\partial {#2}}}}

\DeclareMathOperator{\DO}{DO}
\DeclareMathOperator{\hDO}{{DO_{\hbar}}}

\DeclareMathOperator{\symb}{\sigma}
\DeclareMathOperator{\symbf}{\sigma_{\text{full}}}

\newcommand{\lder}[2]{{\partial {#1}/\partial {#2}}}

\newcommand{\R}[1]{{\mathbb R}^{#1}}
\newcommand{\RR}{\mathbb R}
\newcommand{\CC}{\mathbb C}
\newcommand{\fun}{C^{\infty}}

\newcommand{\al}{{\alpha}}

\newcommand{\D}{{\Delta}}

\newcommand{\e}{{\varepsilon}}

\newcommand{\F}{{\Phi}}

\newcommand{\ps}{{\psi}}

\newcommand{\gt}{{\tilde g}}

\newcommand{\at}{{\tilde a}}

\newcommand{\Ft}{{\tilde F}}
\newcommand{\Ht}{{\tilde H}}
\newcommand{\Yt}{{\tilde Y}}
\newcommand{\Kt}{{\tilde K}}
\newcommand{\Lt}{{\tilde L}}

\newcommand{\tto}{{\linethickness{2pt}
		  \,\begin{picture}(1,0)
                   \put(0,0.26){\line(1,0){0.95}}
                   \put(0,0){$\boldsymbol{\rightarrow}$}
                  \end{picture}
                  }\,
}

\newcommand{\fto}[1]{\stackrel{#1}{\vphantom{\rightrightarrows}\to}}

\newcommand{\dbar}{{\,\mathchar '26 \mkern-11mu d}}
\newcommand{\Dbar}{%
   {{D\mkern-14 mu
   \mathchoice{\raisebox{-2pt}{$\displaystyle \mathchar'26$}}
             {\raisebox{-2pt}{$\mathchar'26$}}
             {\raisebox{-1pt}{$\scriptstyle \mathchar'26$}}
             {\raisebox{-0.5pt}{$\scriptscriptstyle \mathchar'26$}}}
             \mkern 5 mu}}

\DeclareMathOperator{\funh}{\mathit{C^{\infty}_{\hbar}}}

\newcommand{\hpa}{{\hat p_y}^{\,\al}}

\title[Differential operators over a map]{On   differential operators over a map, \\
thick morphisms of supermanifolds, and symplectic micromorphisms}
\author{Ekaterina~Shemyakova}
\address{Department of Mathematics,  University of Toledo, Toledo,  Ohio, 43606, USA}
\email{ekaterina.shemyakova@utoledo.edu}
\author{Theodore~Voronov}
\address{Department of Mathematics,  University of Manchester, Manchester, M13 9PL,  UK  and
Faculty of Physics, Tomsk State University, Tomsk, 634050, Russia}
\email{theodore.voronov@manchester.ac.uk}
\thanks{Research of the first author was partially supported by NSF under grant  1708033. Research of the second author was partially supported by LMS grants.}
\dedicatory{Dedicated to the memory of Alexandre Mikhailovich Vinogradov}

\begin{document}
\begin{abstract}
We recall the notion of a differential operator over a smooth map (in linear and non-linear settings) and consider its versions  such as  formal $\hbar$-differential operators over a map. We study  constructions and examples of such operators, which include    pullbacks by thick morphisms and quantization of symplectic micromorphisms.
\end{abstract}

\maketitle
\tableofcontents

\section{Introduction}\label{sec.intro}

 The notion of a ``differential operator over a map'' (or, in algebraic version, over an algebra homomorphism) is not new. It can be traced to Gabriel~\cite[Expos\'{e} VIIA]{sga3:7}, and can be seen as a natural extension of the algebraic definition of a differential operator on  a scheme or a commutative algebra  by Grothendieck~\cite[\S16.8]{groth:ega4-4}. (For the latter notion, see also Vinogradov~\cite{vinog:1972} and \cite{vinog:vkl1}, and also Koszul~\cite{koszul:crochet85}.)

 However, in spite of its being very ``natural'', this notion is missing from standard texts. Recently constructions appeared such as \emph{thick morphisms} between manifolds or supermanifolds (due to the second author, see~\cite{tv:microformal}) that provide examples of differential operators over   maps, or   versions or modifications of thereof. The purpose of this paper is to review this central notion and its variants, which include  non-linear operators, formal, pseudo- and $\hbar$-(formal, pseudo-) versions. We give constructions and examples of such operators for $\R{n}$ and for (super)manifolds. In particular, we consider operators arising as quantization of symplectic micromorphisms introduced in a recent work by Cattaneo, Dherin and Weinstein and compare them  with pullbacks by quantum  thick morphisms~\cite{tv:oscil,tv:qumicro,tv:microformal}.

 We dedicate this work to  A.~M.~Vinogradov (1938--2019), a remarkable man and mathematician,    friendship with whom we shall always treasure in our memories.

\section{Differential operators over maps and related concepts}\label{sec.do}

Let $\f\co M_1\to M_2$ be a smooth map of differentiable manifolds, which in local coordinates is expressed as $y^i=\f^i(x)$. Then a \emph{differential operator  over a map} $\f$ of \emph{order} $\leq k$ is a linear operator $L\co \fun(M_2)\to \fun(M_1)$ that in local coordinates can be written as
\begin{equation}\label{eq.locdef}
    L(g) = \sum_{|\al|\leq k} L_{\al}(x)\, \p_y^{\,\al}(g)(\f(x))\,.
\end{equation}
(Here $\al$ is a multi-index, $\p_y^{\,\al}=\p_{y^1}^{\,\al_1}\ldots \p_{y^m}^{\,\al_m}$, $|\al|=\al_1+\ldots+\al_m$.)
In other words, we take a function of variables $y^i$, differentiate it with respect to $y^i$   and substitute in the result the variables $y^i$ as functions of $x^a$ (as given by the map $\f$), and then take a linear combination of these derivatives-followed-by-substitution with  the coefficients depending on $x$.

(Everywhere in this section we speak about $\fun$ functions and   maps, but it is equally possible to consider real-analytic or complex-analytic functions or  formal power series.)

\begin{example}
A \emph{vector field over a map} (or \emph{along a map}) gives an example of a differential operator over a map of order $\leq 1$. Such a vector field over a map $\f\co M_1\to M_2$ is defined as a section of $\f^*(TM_2)\to M_1$, i.e. a map $Y\co M_1\to TM_2$ such that $Y(x)\in T_{\f(x)}M_2$. As an operator on functions, in coordinates,
\begin{equation}
    Y=Y^i(x)\,\der{}{y^i}_{|y=\f(x)}\,.
\end{equation}
Alternatively, a vector field $Y$ over a map  $\f$ can be understood as an infinitesimal variation of $\f$, i.e. a ``map'' (depending of a formal parameter $\e$, $\e^2=0$) $\f_{\e}\co M_1\to M_2$, $\f_{\e}(x):=\f(x)+\e Y(x)$. A particular example is the velocity of a parameterized curve, which is a vector field $d\g/dt$ over $\g\co (a,b)\to M$. Another particular example arises when there is a family of maps $\f_t\co M_1\to M_2$ and the derivative $Y_t:=\lder{\f_t}{t}$ is a vector field over $\f_t$ for each $t$. (This vector field over a map $Y_t$ appears in differential geometry for example in Cartan  homotopy formula  for differential forms.)
\end{example}

An algebraic version of the same concept can be formulated as follows. Let $\al\co A\to B$ be an algebra homomorphism of commutative algebras. Then   \emph{differential operators  over an algebra homomorphism $\al$}  (shortly: \emph{d.o.'s} over $\al$) of order $\leq k$ (or $k$th order) are defined inductively by the following conditions. A \emph{differential operator   over   $\al$ of order zero} is a linear map $L\co A\to B$ satisfying
\begin{equation}
    L(aa')= \al(a)\,L(a')
\end{equation}
for all $a,a'\in A$. If $A$, $B$ are algebras with a unit and $\al$ preserves units, one can see that such an $L$ acts as
\begin{equation}
    L(a)=L(a1)=\al(a)L(1)=\al(a)b\,,
\end{equation}
where $b=L(1)\in B$, i.e. $L$ is the combination of the action of the homomorphism $\al$ and a multiplication operator. Now for $k>0$, a linear map $L\co A\to B$ is a \emph{differential operator   over   $\al$ of order $k$} if for all $a,a'\in A$,
\begin{equation}
    L(aa')= \al(a)\,L(a') + L_1(a')
\end{equation}
where $L_1\co A\to B$ is a differential operator   over   $\al$ of order $k-1$ (depending on $a\in A$).
\begin{example}
One can see that  a first order differential operator $L$  over an algebra homomorphism  $\al\co A\to B$ satisfying $L(1)=0$ is nothing but a \emph{derivation over $\al$}, i.e. satisfies the Leibniz rule
\begin{equation}
    L(a_1a_2) = L(a_1)\al(a_2)+ \al(a_1)\,L(a_2)\,,
\end{equation}
and conversely. Such operators $L$ define infinitesimal variations of algebra homomorphisms, $\al_{\e}=\al+\e\,L\co A\to B$. (This is an the algebraic version of a vector field over a smooth map.)
\end{example}

A version of the same definition for   superalgebras includes signs: \emph{$L$ is a differential operator over a superalgebra homomorphism $\al\co A\to B$} between commutative superalgebras if for all $a,a'\in A$
\begin{equation}
    L(aa') = (-1)^{\at_1\tilde L}\al(a)\,L(a') + L_1(a')\,,
\end{equation}
where $L_1$ is order $k-1$.

Like in the usual case, one can show that for algebras of smooth functions the algebraic definition and the coordinate definition give the same notion. If for (super)manifolds $M_1$ and $M_2$ we denote by $\DO^k(M_1\stackrel{\f}{\to}M_2)$ the set of all $k$th order  differential operators over a smooth map $\f\co M_1\to M_2$ and denote
\begin{equation}
    \DO^k(M_1,M_2)=\bigcup_{\f\co M_1\to M_2}\; \DO^k(M_1\stackrel{\f}{\to}M_2)
\end{equation}
and use the similar notation for algebras, then
\begin{align}
    \DO^k(M_1\stackrel{\f}{\to}M_2)&= \DO^k(\fun(M_2)\stackrel{\f^*}{\to}\fun(M_1))\,,\\ \DO^k(M_1,M_2)&=\DO^k(\fun(M_2),\fun(M_1))\,.
\end{align}

We shall refer to the map $\f\co M_1\to M_2$ for an operator $L\in \DO(M_1,M_2)$  as the \emph{core} or \emph{carrier} of $L$.

One can  check that differential operators over maps with matching source and target can be composed, so if $L\in \DO^k (M_1\stackrel{\f_{21}}{\to}M_2)$ and $K\in \DO^{\ell}\, (M_2\stackrel{\f_{32}}{\to}M_3)$, then
\begin{equation}
    L\circ K\in \DO^{k+\ell}\, (M_1\stackrel{\f_{32}\circ \f_{21}\vphantom{\int_a^b}}{\to}M_3)\,.
\end{equation}
Therefore we obtain a category whose arrows are differential operators over maps. Denote it $\DO$. It contains as a subcategory the (opposite to the) usual category of (super)manifolds and smooth maps, if one identifies a map $\f\co M_1\to M_2$ with a zero-order differential operator over itself, $L=\f^*$.

The category $\DO$ is not additive, as one cannot always add elements of $\DO(M_1,M_2)$, unless they are   over the same map $\f$. So the category $\DO$ is not a so straightforward   generalization of the usual algebra of differential operators $\DO(M)$ for a fixed manifold $M$ (which are   operators over the identity map). Still, it makes sense to ask about ``generators'' of this category  with respect to compositions and sums\,---\, similarly to   the description of the algebra of (polynomial) differential operators on $\R{n}$ as the Weyl algebra. It seems that as such generators of this ``semi-additive'' category $\DO$ one can take: (1) all   vector fields over maps $Y$, for all $\f\co M\to N$\,; (2) all pull-backs $\f^*$ by   maps  $\f\co M\to N$ for all $M$, $N$\,;   and (3) all operators of multiplication by   functions $f\in\fun(M)$,  for all $M$, satisfying  the ``Heisenberg-type'' relation
\begin{equation}
    Y\circ g=(-1)^{\Yt\gt}\f^*(g)\circ Y+Y(g)
\end{equation}
and the relation $\f^*\circ g=\f^*(g)\f^*$\,.

It is not of great difficulty to generalize the above definitions to the case of operators between  modules over commutative superalgebras or, in the differential-geometric setting, to operators acting on sections of (super) vector bundles. In the expression in local coordinates~\eqref{eq.locdef}, this would amount to consider matrix coefficients. We shall not go in  this direction further. Instead we shall discuss two particular variations of our theme: non-linear operators and $\hbar$-formal operators.

We say that a mapping  $L\co \fun(M_2)\to \fun(M_1)$ is a \emph{non-linear differential operator over a smooth map} $\f\co M_1\to M_2$  of order $\leq k$  if  $L$ sends a function $g\in \fun(M_2)$ to a function $L(g)=f\in \fun(M_1)$ that in  local coordinates   is expressed as a polynomial in partial derivatives $\p_y^{\,\al}g$ with $|\al|\leq k$ evaluated at $y=\f(x)$ with the coefficients depending on $x$\,:
\begin{equation}\label{eq.nonlindef}
    L(g)(x)=P\left(x,\p g,\p^2 g, \ldots, \p^k g\right)_{|y=\f(x)}\,,
\end{equation}
where by $\p^rg$ we denote  the whole collection of partial derivatives $\p_y^{\,\al}g(y)$ for all $\al$ with $|\al|=r$. The right-hand side of~\ref{eq.nonlindef} is a polynomial in $\p g(\f(x)),\p^2 g(\f(x)), \ldots, \p^k g(\f(x))$.  (Linear operators considered above are of course a particular case,   when the polynomial $P$ in~\ref{eq.nonlindef} is linear in the derivatives.)

One can say this equivalently by  introducing a   bundle   $J^k(M_1\stackrel{\f}{\to}M_2)$ over the manifold $M_1$, as the pull-back bundle
\begin{equation*}
    J^k(M_1\stackrel{\f}{\to}M_2):=\f^*\bigl(J^k(M_2)\bigr)
\end{equation*}
of the jet bundle $J^k(M_2)=J^k(M, \RR)$ over $M_2$. Then a non-linear differential operator of order $\leq k$ over $\f$ is a fiberwise-polynomial function on the total space $J^k(M_1\stackrel{\f}{\to}M_2)$. Again, as in the linear case, this can be further generalized to sections of (possibly non-linear) fiber bundles instead of scalar functions. Everywhere where we say manifold, we actually can say supermanifold (or graded manifold, see~\cite{tv:gradedmicro}).

Coming back to the linear case, an important variation of the definition above is as follows.

Let $\hbar$ be a formal parameter to which we shall refer to as ``Planck's constant''. Consider smooth functions on (super)manifolds that depend on $\hbar$ as formal power series (with non-negative powers only).  For them we use the notation $\funh(M)$. (Other classes can be also useful, such as e.g. formal oscillatory exponentials with coefficients from $\funh(M)$.) Now $\CC[[\hbar]]$ is the ground ring instead of $\CC$. (From this point, it is convenient to work with complex-valued functions.)

For a given map $\f\co M_1\to M_2$, we say that a linear (i.e. $\CC[[\hbar]]$-linear) operator
 \begin{equation*}
    L\co \funh(M_2\to \funh(M_1)
 \end{equation*}
is an \emph{$\hbar$-differential operator over $\f$ of order $\leq k$} (abbreviation: $\hbar$-d.o.) if for every $g\in \funh(M_2)$,
\begin{equation}\label{eq.hdef}
     L\circ g-(-1)^{\gt\Lt} \f^*(g)\circ L = -i\hbar L_1\,,
\end{equation}
where $L_1\co \funh(M_2\to \funh(M_1)$ is an $\hbar$-differential operator over $\f$ of order $\leq k-1$ and all $\hbar$-differential operator over $\f$ of order $\leq 0$ are zero.
\begin{example}
As before, we can see that an $\hbar$-d.o. of order $0$ has the form $L=f_0\cdot \f^*$, where $f_0=L(1)$, so at this stage no difference arises. However, for an $\hbar$-d.o. of order $\leq 1$, we can deduce from the definition that every such operator has the form
\begin{equation*}
    L=-i\hbar D + f_0\cdot\f^*\,,
\end{equation*}
where  $D\co \funh(M_2)\to \funh(M_1)$ is a derivation over $\f^*$, i.e. a vector field over $\f$, and $f_0\in \funh(M_1)$, $f_0=L(1)$.
\end{example}

In general, we can deduce that in local coordinates an $\hbar$-differential operator over a map $\f$ of order $\leq k$ has the form
\begin{equation}\label{eq.locdefh}
    L(g) = \sum_{|\al|\leq k} L_{\al}(x)\, (-i\hbar\p_y)^{\,\al}(g)(\f(x))\,,
\end{equation}
where the coefficients $L_{\al}(x)$ are power series in $\hbar$. In other words, we have a special case of~\eqref{eq.locdef} with an extra condition that every partial derivative $\lder{}{y^a}$ carries a factor of $-i\hbar$. Denote by ${\hat p_y}^{\,\al}$ the operator $\f^*\circ (-i\hbar \p_y)^{\,\al}$. (Warning: $\hpa$ is \emph{not} a product!) Then
\begin{equation}\label{eq.locexpr}
    L  = \sum_{|\al|\leq k} L_{\al}(x)\, \hpa
\end{equation}
is a general form of an $\hbar$-d.o. over $\f$. This expression of course depends on a choice of local coordinates, and it is not difficult to deduce a transformation law for the symbols $\hpa$, as well as commutation relations with functions on $M_2$.

There is an important observation similar to the one made in~\cite{shemy:koszul}: it is possible to introduce a grading into the space of $\hbar$-d.o.'s over a map (besides filtration given by order). We define the \emph{degree} of an $\hbar$-differential operator over $\f$ by the following rules:
\begin{equation}\label{eq.defdeg}
    \deg \hpa:=|\al|\,, \quad \deg \hbar=1\,, \quad \deg f(x)=0\,
\end{equation}
(compare with the definition of  total degree of an $\hbar$-d.o. on a manifold $M$ in~\cite[\S3.2]{shemy:koszul}).

\begin{proposition}\label{prop.degree}
The degree defined by~\eqref{eq.defdeg} does not depend on a choice of local coordinates on $M_1$ and $M_2$.
\end{proposition}
\begin{proof}
We have $\hpa=\f^*\circ (\hat p_1^{\al_1}\ldots \hat p_m^{\al_m})$, where $\hat p_i=-i\hbar \lder{}{y^i}$. Under a change of coordinates, each operator $\hat p_i$ becomes a linear combination (with coefficients independent of $\hbar$) of similar operators relative   ``new'' coordinate system. When we move the coefficients to the left, we use the Leibniz rule and at each step we lose one operator $\hat p_{i'}$ but gain one factor of $-i\hbar$. So the total degree does not change.
\end{proof}

Strictly speaking, degree is well-defined only on operators ``of finite type'', i.e. those whose coefficients are polynomials in $\hbar$. By taking infinite sums of such operators $L_{[k]}$, $\deg L_{[k]}=k$, of all degrees $k=0, 1, 2, \ldots $,
\begin{equation}\label{eq.formal}
    L=L_{[0]} +L_{[1]}+ L_{[2]}+\ldots
\end{equation}
we arrive at the notion of \emph{formal $\hbar$-differential operators over a map $\f$}. (In the next section we shall push this further   to obtain  ``pseudodifferential operators over a smooth map''.) We shall denote the space of all formal $\hbar$-d.o.'s over $\f\co M_1\to M_2$ by $\hDO(M_1\fto{\f}M_2)$ and the space of all  $\hbar$-d.o.'s over all maps from $M_1$ to $M_2$ by $\hDO(M_1,M_2)$,
\begin{equation*}
    \hDO(M_1,M_2)= \bigcup_{\f\co M_1\to M_2}\;\hDO(M_1\fto{\f}M_2)\,.
\end{equation*}
By $\hDO^{\![k]}(M_1\fto{\f}M_2)$ and  $\hDO^{\![k]}(M_1,M_2)$ we denote the corresponding spaces of operators of degree $k$, so
\begin{equation*}
    \hDO\,(M_1,M_2)=\prod_{k=0}^{+\infty} \hDO^{\![k]}(M_1,M_2)
\end{equation*}

The same argument as in the proof of Proposition~\ref{prop.degree} shows that modulo $\hbar$, the operators $\hpa$ behave under a change of coordinates as products of commutating variables. Indeed, we have a product of operators $\hat p_i$ followed by the substitution $y=\f(x)$. If we change coordinates on $M_2$, we obtain that for a single such operator,
\begin{equation*}
    \hat p_i=\der{y^{i'}}{y^i}(y)\,\hat p^{i'}\,.
\end{equation*}
Hence, for the transformation formula for a product   $\hat p_{i_1}\ldots p_{i_k}$, we need to move the coefficients of the Jacobi matrix to the left of all ``new'' $\hat p_{i'}$. Each time, by using the commutation relation, we gain an extra term proportionate to $\hbar$. Hence modulo $\hbar$ (and taking into account the substitution $y=\f(x)$) we arrive at the transformation law of the product of commuting variables $p_i$, where for each variable we have
\begin{equation*}
    p_i= \der{y^{i'}}{y^i}\bigl(\f(x)\bigr)\,p_{i'}\,.
\end{equation*}
This is exactly the transformation law for fiber coordinates in the bundle $\f^*(T^*M_2)\to M_1$. We have arrived at the following statement.
\begin{theorem}
To every formal $\hbar$-differential operator $L$  over a map $\f\co M_1\to M_2$ we can canonically assign a function $H=\symb(L)$ on the bundle $\f^*(T^*M_2)$ by setting $\hbar=0$ in the coefficients in~\eqref{eq.locexpr} and replacing the operators $\hpa$ by the monomials $p^{\al}$, where $p_i$ are fiberwise coordinates. The function $\symb(L)$ is a formal power series in $p_i$ and a polynomial on $p_i$ if $L$ is an $\hbar$-d.o. over $\f$. \qed
\end{theorem}

For the simplicity of notation we use $\fun(T^*M)$ and similar  also for functions that are formal power series along the fibers.

The function $\symb(L)$, a power series or a polynomial in the fiber variables on $\f^*T^*M_2$, is called the \emph{principal symbol} of $L$. It is a new notion. (Compare with the definitions of the principal symbol of a formal $\hbar$-differential operator  on a   manifold   in \cite{shemy:koszul} and   in \cite{tv:microformal}.)

Suppose we have maps $\f_{21}\co M_1\to M_2$ and $\f_{32}\co M_2\to M_3$ and operators
\begin{equation*}
    L\in \hDO(M_1\fto{\f_{21}}  M_2)\,, \quad K\in \hDO(M_2\fto{\f_{32}}  M_3)\,.
\end{equation*}
We have the composition
 \begin{equation*}
    L\circ K\in \hDO(M_1\fto{\f_{32}\circ\f_{21}}  M_3)\,.
 \end{equation*}
What can be said about the principal symbols? We know what to expect in the classical situation of a single manifold. To be able to say that ``the principal symbols multiply'', we need actually to introduce the corresponding multiplication. The problem is that they are functions on different bundles. However, they possess nice functorial properties. Namely, if $H\in \fun(\f_{21}^*T^*M_2)$ and $F\in \fun(\f_{32}^*T^*M_3)$, there are the  \emph{pull-back} $\f_{21}^*(F)\in \fun(\f_{31}^*T^*M_3)$ and the \emph{push-forward} ${\f_{32}}_*(H)\in \fun(\f_{31}^*T^*M_3)$. In a self-explanatory notation for the position and momentum variables, $H=H(x_1,p_2)$, $F=F(x_2,p_3)$ and
\begin{equation*}
    {\f_{32}}_*(H)=H\Bigl(x_1,\der{x_3}{x_2}p_3\Bigr)\quad \text{and} \quad \f_{21}^*(F)=F\bigl(x_2(x_1),p_3\bigr)\,.
\end{equation*}
We define the \emph{product} of functions $H\in \fun(\f_{21}^*T^*M_2)$ and $F\in \fun(\f_{32}^*T^*M_3)$ to be a function $HF=FH$ (in the supercase $HF=FH(-1)^{\Ft\Ht}$) on the bundle $\f^*_{31}T^*M_3$, where $\f_{31}:=\f_{32}\circ \f_{21}$, given by
\begin{equation}\label{eq.prodsymb}
    H\cdot F:={\f_{32}}_*(H)\,\f_{21}^*(F)\,,
\end{equation}
where at the right-hand side is the usual product of functions on $\f^*_{31}T^*M_3$.

\begin{theorem}
For formal $\hbar$-differential operators over maps,
\begin{equation}\label{eq.symbprod}
    \symb(L\circ K)=\symb(L)\cdot\symb(K)\,.
\end{equation}
\end{theorem}
\begin{proof}
Directly by the definitions of the principal symbol and the product~\eqref{eq.prodsymb}.
\end{proof}

Suppose we have a commutative diagram of smooth maps:
\begin{equation}\label{eq.comdiag}
    \begin{CD}
    M_1 @>{\f_{21}}>> M_2 \\
    @V{\ps_{31}}VV @VV{\ps_{42}}V \\
    M_3 @>>{\f_{43}}> M_4
    \end{CD}\,,
\end{equation}
so $\ps_{42}\circ \f_{21}=\f_{43}\circ \ps_{31}$\,, and suppose we have formal  $\hbar$-d.o.'s $L_{12}$ over $\f_{21}$, $L_{34}$ over $\f_{43}$, $K_{13}$ over $\ps_{31}$, and $K_{24}$ over $\ps_{42}$. Since the diagram~\eqref{eq.comdiag} is commutative, the compositions $L_{12}\circ K_{24}$ and $K_{13}\circ L_{34}$ are defined over the same map, so can be compared. Consider the difference
\begin{equation}
    \D=L_{12}\circ K_{24} -(-1)^{\Lt\Kt} K_{13}\circ L_{34}
\end{equation}
(we assume that the parities agree so that $\tilde L_{12}=\tilde L_{34}=:\Lt$ and $\tilde K_{24}=\tilde K_{13}=:\Kt$). It is not particularly interesting if we do not assume any relation between the operators. Suppose further that
\begin{equation}
    {\ps_{42}}_*\left(\symb(L_{12})\right)= \ps_{31}^*\left(\symb(L_{34})\right) \quad \text{and} \quad
    \f_{21}^*\left(\symb(K_{24})\right)= {\f_{43}}_*\left(\symb(K_{13})\right)\,.
\end{equation}
It follows that $\symb(\D)=0$, by the commutativity of the product of symbols. Hence $\D$ is divisible by $\hbar$. We can define an analog of the Poisson bracket, by
\begin{equation}\label{eq.ourbrack}
    \{H_{12}, H_{34}\,;\,F_{24},F_{13}\}:=\symb\left(\frac{i}{\hbar}\D\right)\,,
\end{equation}
where we denoted $H_{12}=\symb(L_{12})$, $H_{34}=\symb(L_{34})$, and
$F_{24}=\symb(K_{24})$, $F_{13}=\symb(K_{13})$\,. We hope to investigate   this operation elsewhere.

\section{Constructions and examples}\label{sec.ex}

In this section we consider constructions leading to (formal, $\hbar$-) differential operators over maps. We note that in the same way as familiar differential operators on $\R{n}$ or on a manifold, differential operators over maps can be defined by integral formulas using various forms of ``full symbol calculus''. Consider first the simplest case of maps between Cartesian spaces.
Let $\f\co \R{n_1}\to \R{n_2}$ be a smooth map. Then by the definition of a formal $\hbar$-differential operator over $\f$, every such operator $L\co \funh(M_2)\to \funh(M_1)$ can be expressed as
\begin{equation}\label{eq.int}
    L(g)(x_1)=\int_{\R{2n_2}} dx_2\dbar p_2\; e^{\frac{i}{\hbar}(\f(x_1)-x_2) p_2}\, H_{\hbar}(x_1,p_2)\,g(x_2)\,.
\end{equation}
Here we use standard notations such as $\dbar p$ for denoting coordinate volume element normalized so that it contains all numerical factors depending on dimension arising in inversion formulas for $\hbar$-Fourier transform. (In the supercase, we use similar notation e.g. $\Dbar p$ etc.) Here $H_{\hbar}(x_1,p_2)$ is a formal series (in $p_i$ and $\hbar$) of the form
\begin{equation}\label{eq.ampl}
    H_{\hbar}(x_1,p_2) = \sum_{k=0}^{+\infty} \left(H^{i_1\ldots i_k}_0(x_1)p_{i_1}\ldots p_{i_k}+(-i\hbar)H^{i_1\ldots i_{k-1}}_1(x_1)p_{i_1}\ldots p_{i_{k-1}}+\ldots + H_k^0(x_1)\right)\,.
\end{equation}
The coefficients $H^{i_1\ldots i_k}(x_1)$ etc. do not depend on $\hbar$. We refer to the function $H_{\hbar}(x_1,p_2)$ as the \emph{full symbol} of $L$. If we need a notation, we shall write $\symbf(L)$. One can find the full symbol of $L$ by the formula
\begin{equation}\label{eq.fullsymb}
    \symbf(L)=e^{-\frac{i}{\hbar}\f(x_1)p_2}L(e^{\frac{i}{\hbar}x_2p_2})\,.
\end{equation}

It is clear that instead of formal power series one can consider functions from different classes as long as the integral makes sense and this will give various types of ``$\hbar$-pseudodifferential operators over a map $\f$''. They all will be of course particular examples of Fourier integral operators~\cite{hoer:fio1-1971}.

The full symbol given by with~\eqref{eq.ampl} and \eqref{eq.fullsymb} and the principal symbol defined in the previous section are related by
\begin{equation}\label{eq.twosymb}
    \symb(L)=\symbf(L)_{|\hbar=0}\,,
\end{equation}
or, in terms of the expansion~\eqref{eq.ampl},
\begin{equation}\label{eq.prinsymb}
    \symb(L)=\sum_{k=0}^{+\infty}  H^{i_1\ldots i_k}_0(x_1)p_{i_1}\ldots p_{i_k}\,.
\end{equation}

Everything above can be done in the supercase, replacing $\R{n_1}\to \R{n_2}$ by $\R{n_1|m_1}\to \R{n_2|m_2}$. It makes no principle difference, so we do not dwell on that.

Generalization from   Cartesian spaces to (super)manifolds of formulas~\eqref{eq.int}--\eqref{eq.fullsymb} can be done in two different ways. In the first approach, one can simply consider integral formulas such as~\eqref{eq.int} in   coordinate domains and require  that they specify an operator independent on a choice of coordinates. If we write the same formula as~\eqref{eq.int} for manifolds $M_1$ and $M_2$ and a map $\f\co M_1\to M_2$ as
\begin{equation}\label{eq.intman}
    L(g)(x_1)=\int_{T^*M_2} dx_2\dbar p_2\; e^{\frac{i}{\hbar}(\f(x_1)-x_2) p_2}\, H_{\hbar}(x_1,p_2)\,g(x_2)\,,
\end{equation}
then the function $H_{\hbar}(x_1,p_2)$ will not be an invariantly-defined function on $\f^*T^*M_2$\footnote{But the function $H_{0}(x_1,p_2)$ obtained by setting $\hbar=0$ will be a well-defined function on $\f^*T^*M_2$, hence the ``non-invariance'' of the full symbol can be seen as ``quantum corrections''.}, but instead will have a non-trivial transformation law, similar with the transformation law for full symbols of (pseudo)differential operators on manifolds. 
Another approach can be based on choosing an extra structure on manifolds in question such as a connection and a volume element and/or a metric. Then the ``non-invariance'' of integral formulas will be packed instead of a dependence on a choice of coordinates into a dependence of a choice such an extra structure (e.g. connection). Let us give the corresponding formulas. (Note that there may be slightly non-equivalent ways for writing them, and we use full symbol calculus for pseudodifferential operators on Riemannian manifolds built in~\cite{tv:forms} as prototype.)

Let $H\in \funh(\f^*T^*M_2)$ for a map $\f\co M_1\to M_2$. Define the  operator $\hat H\co \funh(M_2)\to \funh(M_1)$ by the formula
\begin{equation}\label{eq.intconn}
    (\hat H g)(x_1)= \int\limits_{T^*_{\f(x_1)}M_2\times T_{\f(x_1)}M_2}
    dv_2\dbar p_2\; e^{-\frac{i}{\hbar}v_2p_2}\, H(x_1,p_2)\,g(\exp_{\f(x_1)}v_2)\
\end{equation}
We do not place $\hbar$ explicitly in the notation for $H$, in part for   distinction with $H_{\hbar}$ in formula~\eqref{eq.intman} (but $H$ still is a power series in $\hbar$).  Here $\exp$ is the exponential mapping defined by a connection on $M_2$. 
By a change of variables $x_2=\exp_{\f(x_1)}v_2$, it is possible to rewrite~\eqref{eq.intconn} also as
\begin{equation}\label{eq.intconn2}
    (\hat H g)(x_1)= \int_{M_2\times T^*_{\f(x_1)}M_2}
    dx_2\dbar p_2\; \mu(x_1,x_2)\,e^{\frac{i}{\hbar}\exp^{-1}_{\f(x_1)}(x_2)\,p_2}\, H(x_1,p_2)\,g(x_2)
\end{equation}
or
\begin{equation}\label{eq.intconn3}
    (\hat H g)(x_1)= \int_{T^*M_2}
    dx_2\dbar p_2\; \mu(x_1,x_2)\,e^{-\frac{i}{\hbar}\exp^{-1}_{x_2}(\f(x_1))\,p_2}\, H(x_1,\tau(\f(x_1),x_2)p_2)\,g(x_2)
\end{equation}
Here $\mu(x_1,x_2)$ is some Jacobian function arising from a change of variables and $\tau(x_2,x_2')$ is the parallel translation along the geodesic joining $x_2$ and $x_2'$.

Define a class of functions on $\f^*T^*M$ that are power series in momentum variables and $\hbar$ together. In particular, this class includes polynomials. We will continue to use the notation $\funh(\f^*T^*M)$ meaning this class.

Strictly speaking, the exponential mapping ceases being invertible for large tangent vectors, so one may wish to insert some bump function into the integrals to take care of that (e.g. like it is done in~\cite{tv:forms}). However, for the functions $H$ that we consider it is not necessary, since their $\hbar$-Fourier transform is supported at the graph $x_2=\f(x_1)$.

\begin{proposition}
For every $H\in \funh(\f^*T^*M_2)$, the
operator  given by \eqref{eq.intconn}, \eqref{eq.intconn2} \eqref{eq.intconn3} is a formal $\hbar$-differential operator over $\f$.
\end{proposition}
\begin{proof}
Use normal coordinates centered at $x_2=\f(x_1)$ and express the integral in these coordinates.
\end{proof}

\begin{example}
Consider an integral operator defined by the formula
\begin{equation}\label{eq.qthick}
    L(g)(x_1)=\int_{T^*M_2} dx_2\dbar p_2\; e^{\frac{i}{\hbar}(S(x_1,p_2))-x_2p_2}g(x_2)\,,
\end{equation}
where $S(x_1,p_2)$ is a power series in $p_2$. We call $S$ a (quantum)   \emph{generating function}. Compare with~\eqref{eq.intman}: in~\eqref{eq.qthick} there is no ``amplitude'' $H(x_1,p_2)$ in front of the oscillating exponential, but instead of $\f(x_1)p_2$ in the exponential there is $S(x_1,p_2)$. If we   express $S(x_1,p_2)$ as
\begin{equation}\label{eq.s}
    S(x_1,p_2)=S^0_{\hbar}(x_1)+\f^i_{\hbar}(x_1)p_{2i}+ S^{+}_{\hbar}(x_1,p_2)\,,
\end{equation}
where $S^{+}_{\hbar}(x_1,p_2)$ contains terms of order $\geq 2$ in $p_2$, it becomes possible to rewrite~\eqref{eq.qthick} in the same form as~\eqref{eq.intman} as its special case.
\end{example}

\begin{theorem}[\cite{tv:qumicro}]
The operator defined by~\eqref{eq.qthick} is a formal $\hbar$-differential operator over a map $\f\co M_1\to M_2$, of the form
\begin{equation}\label{eq.explic}
    L=e^{\frac{i}{\hbar}S^0(x_1)}\left(e^{\frac{i}{\hbar} S^{+}(x_1,\frac{\hbar}{i}\der{}{x_2})}\right)_{x_2=\f_{\hbar}(x_1)}\,.
\end{equation}
\end{theorem}

Here $\f_{\hbar}$ is an  \emph{$\hbar$-perturbation  of a smooth map} $\f\co M_1\to M_2$. It is given in local coordinates by formulas $y^i=\f^i_{\hbar}(x)$, where $\f^i_{\hbar}(x)=\f_0^i(x)+\hbar \f_1^i(x)+\ldots $ is a formal power series such that $\f^i_0(x)=\f^i(x)$ specify the initial map $\f$, and these local descriptions transform appropriately on the intersections of coordinate charts. With an abuse of language we   speak simply of a map $\f_{\hbar}\co M_1\to M_2$   ``depending on $\hbar$''.

Operators~\eqref{eq.qthick} were introduced in~\cite{tv:oscil} as ``pullbacks by quantum thick morphisms''.
A distinctive feature of such operators   is that in the classical limit obtained by the stationary phase method, see~\cite{tv:microformal}, they give pullbacks by classical \emph{thick morphisms} $\F^*$, introduced in~\cite{tv:nonlinearpullback}, which are formal non-linear differential operators over smooth maps,
\begin{equation*}
    \F^*(g)=\F^*_{[0]}(g)+ \F^*_{[1]}(g)+ \F^*_{[2]}(g)+\ldots
\end{equation*}
where each summand $\F^*_{[k]}(g)$ is a non-linear differential operator over a map applied to $g$ of order $k$ in $g$, expansion over $k$, and among all such nonlinear operators pullbacks by thick morphisms are distinguished because they are \emph{non-linear algebra homomorphisms}. By definition~\cite{tv:microformal}, it is such a (formal) map of algebras that its derivative for every element is a usual algebra homomorphism.

\begin{theorem}
If a formal non-linear operator $L\co \fun(M_2)\to \fun(M_1)$   is a non-linear algebra homomorphism, then $L=\F^*$, the pullback by some (unique) thick morphism $\F\co M_1\tto M_2$.
\end{theorem}

This statement has been   recently proved by H.~Khudaverdian. (It was conjectured  by the second author, see~\cite{tv:microformal}.)

It is possible to specify quantum thick morphisms by an ``invariant'' generating function (depending on  a connection). The corresponding formula will be
\begin{equation}\label{eq.invthick}
    L(g)(x_1)=\int\limits_{T^*_{\f(x_1)}M_2\times T_{\f(x_1)}M_2}
    dv_2\dbar p_2\; e^{\frac{i}{\hbar}(S^0(x_1)+\bar S^{(+)}(x_1,p_2)-v_2p_2)}\,g(\exp_{\f(x_1)}v_2)\
\end{equation}
Here we had to explicitly identify the carrier map $\f\co M_1\to M_2$. The function $\bar S^{(+)}(x_1,p_2)$ is a global function on $\f^*T^*M_2$.

The last example that  we shall consider is ``quantization of symplectic micromorphisms'' as introduced in~\cite{cattaneo-dherin-weinstein:quantization}. A ``symplectic micromorphism'' in the terminology of Cattaneo, Dherin and Weinstein is a morphism between ``symplectic microfolds''; a symplectic microfold is a germ of a symplectic manifold at a Lagrangian submanifold, by Weinstein's symplectic tubular neighborhood theorem it can be identified with the germ of a cotangent bundle. A symplectic micromorphism between such germs is defined as   (the germ of) a  canonical relation which is ``close'' to the relation corresponding to a map of bases. ``Close'' basically means that it can be specified by a generating function of the type $S(x_1,p_2)$, i.e. exactly the same as in the definition of thick morphisms (classical or quantum). We can say that symplectic micromorphisms and thick morphisms are very close, the difference being like between a germ and a jet (and also in the presence of $S^0(x_1)$ for thick morphisms). A ``quantization of a  symplectic micromorphism'' is a linear integral operator of the form very close to~\eqref{eq.invthick}:
\begin{equation}\label{eq.qsympl}
    L(g)(x_1)=\int\limits_{T^*_{\f(x_1)}M_2\times T_{\f(x_1)}M_2}
    dv_2\dbar p_2\; e^{\frac{i}{\hbar}(\bar S^{(+)}(x_1,p_2)-v_2p_2)}\,H(x_1,p_2)\,g(\exp_{\f(x_1)}v_2)\,.
\end{equation}
(In~\cite{cattaneo-dherin-weinstein:quantization} they consider operators acting on half-densities, but this makes no essential difference. Also, $\exp$ is not necessarily defined by a connection.) It is assumed that $\bar S^{(+)}(x_1,p_2)$ as a function of $p_2$ has zero of order two at $p_2$. We see as the main difference with pull-backs by thick morphisms the presence of the function $H(x_1,p_2)$, which is a genuine function on $\f^*T^*M_2$. Also, no term $S^0(x_1)$ in the exponential. The particular case of $\bar S^{(+)}(x_1,p_2)=0$ is called in~\cite{cattaneo-dherin-weinstein:quantization} ``quantization of cotangent lift''. It is in fact the same as an $\hbar$-differential operator over $\f$ written in an integral firm as \eqref{eq.intman} or \eqref{eq.intconn}. Moreover, a closer look shows that the class of operators obtained by formula~\eqref{eq.qsympl} is not different from the   class of operators over a map. If $\hbar$ is treated as a formal parameter, one can see that only the Taylor expansions of $\bar S^{(+)}(x_1,p_2)$ and $H(x_1,p_2)$ play a role and we have the following statement.  (See also remark below.)

\begin{theorem}
If $\hbar$ is regarded as a formal parameter, then the class of operators obtained as ``quantization of symplectic micromorphisms''  coincides with the class of all formal $\hbar$-differential operators over smooth maps.
\end{theorem}


It we take the viewpoint that thick morphisms are generalizations of ordinary maps, then one may consider ``differential operators over thick morphisms''. The practical difference is what emerges as their symbols: if we have for example
\begin{equation*}
    L(g)(x_1)= \left(e^{\frac{i}{\hbar} S^{+}(x_1,\frac{\hbar}{i}\der{}{x_2})}
    \,H\Bigl(x_1,\frac{\hbar}{i}\der{}{x_2}\Bigr)\,g(x_2)\right)_{x_2=\f_{\hbar}(x_1)}\,,
\end{equation*}
then it is either $H(x_1,p_2)$, which may be polynomial in $p_2$, or   $H(x_1,p_2)\,e^{\frac{i}{\hbar} S^{+}(x_1,p_2)}$.

\def\cprime{$'$} \def\cprime{$'$}

\end{document}